\documentclass[12pt]{article}

\usepackage{enumerate}

\usepackage[utf8]{inputenc}
\usepackage{amsmath}
\usepackage{mathtools}
\usepackage{amsthm}
\usepackage{amsfonts}
\usepackage{color}
\usepackage{authblk}

\usepackage{amssymb}
\usepackage[margin=1in]{geometry}
\usepackage[all]{xy}
\usepackage[english]{babel}
\usepackage{graphicx}
\usepackage{arydshln}
\usepackage{tikz}
\usepackage{tikz-cd}
\usepackage{pgfplots}
\usepackage{hyperref}
\usepackage{fancyhdr}
\usepackage{dsfont}
\usepackage{mathrsfs}
\usetikzlibrary{matrix,arrows,decorations.pathmorphing}
\usepackage{wrapfig}

\newtheorem{thm}{Theorem}[section]
\newtheorem{prop}[thm]{Proposition}

\newtheorem{lemma}[thm]{Lemma}
\newtheorem{remark}[thm]{Remark}
\newtheorem{example}[thm]{Example}

\newtheorem*{dfn}{Definition}

\newcommand{\Z}{
	\mathbb{Z}
}

\newcommand{\R}{
	\mathbb{R}
}

\newcommand{\diam}{\mathrm{diam}}
\newcommand{\dgh}{d_\mathrm{GH}}
\newcommand{\thick}[1]{\mathrm{T}_{#1}}

\definecolor{darkblue}{rgb}{0.0, 0.0, 0.8}
\definecolor{darkred}{rgb}{0.8, 0.0, 0.0}
\definecolor{darkgreen}{rgb}{0.0, 0.8, 0.0}

\begin{document}
\title{The Distortion of the Reeb Quotient Map on Riemannian Manifolds\footnote{This work was partially supported by NSF grants IIS-1422400 and  CCF-1526513.}}

\author[1]{Facundo M\'emoli}
\author[2]{Osman Berat Okutan}
	\affil[1]{Department of Mathematics and Department of Computer Science and Engineering, The Ohio State University. \texttt{memoli@math.osu.edu}}
	\affil[2]{Department of Mathematics, The Ohio State University. \texttt{okutan.1@osu.edu}}
	
	\date{\today}
	\maketitle
    
\begin{abstract}
Given a metric space $X$ and a function $f: X \to \R$, the Reeb construction gives metric a space $X_f$ together with a quotient map $X \to X_f$. Under suitable conditions $X_f$ becomes a metric graph and can therefore be used as a graph approximation to $X$. The Gromov-Hausdorff distance from $X_f$ to $X$ is bounded by the half of the metric distortion of the quotient map. In this paper we consider the case where $X$ is a compact Riemannian manifold and $f$ is an excellent Morse function. In this case we provide bounds on the distortion of the quotient map which involve the first Betti number of the original space and a novel invariant which we call \emph{thickness}.
\end{abstract}

\tableofcontents

\section{Introduction}
    
    Every compact metric space can be approximated by finite metric spaces under the Gromov-Hausdorff distance \cite{bbi01}. Such an approximation can be considered as a simpler and faithful representation of the original space. However, if the original space is a length space, we may ask more structure in our approximation; for example we may require the approximating space to be a length space too. The simplest length spaces are metric graphs. It is known that every compact connected length space can be approximated by finite metric graphs \cite[Proposition 7.5.5]{bbi01}.
    
    A successful approximation method should provide us two pieces of information: A bound controlling the \textit{complexity} of the approximating space, and a bound for the \textit{faithfulness} of the approximation. For example, in the case of approximating compact metric spaces via finite ones, we may want to have control on the cardinality as a measure of complexity, and a control on the Gromov-Hausdorff distance to the original space as a measure of faithfulness. In this paper, as we consider approximation via metric graphs, we use the genus of the graph as the measure of complexity and use the Gromov-Hausdorff distance as the measure of faithfulness. 
    
    Although every length space can be approximated by metric graphs, the resulting graph can be quite complicated in the sense that it may contain a large number of vertices, edges and its genus can get quite large. One way of obtaining a tame graph from a space is the Reeb construction \cite{reeb}, which assigns a graph $X_f$ to a real valued function $f: X \to \R$. $X_f$ is tame in the sense that its genus is less than or equal to the first Betti number of $X$ and its vertices are obtained through the critical values of $f$. 
    
    The Reeb graph construction was introduced in \cite{reeb} to obtain a graph for studying the level sets of a Morse function.  Then the question becomes controlling the distance between a space and a Reeb graph obtained from it. This problem is considered in \cite{chs15} for a length space $X$ which is known to be close to a possibly complicated graph. In \cite{z06}, the case where the original space is a closed surface is studied. 
    
    In this paper, we generalize the result of \cite{z06} to arbitrary $n$-dimensional closed Riemannian manifolds. In order to do this we first need to introduce a new metric invariant $\thick{f}$ associated to any filtered metric space $f: X \to \R$ (see section \ref{invariant}) which we refer to $\thick{f}$ as the \emph{thickness} of $f:X\rightarrow \R$. This invariant gives a quantitative measure of how the volume of the level sets is distributed with respect to their diameters. This is less important in the case of dimension two since then the volume of a level set becomes its length which is always greater than or equal to the diameter. Our main result is the following:
    
\begin{thm}\label{main}\footnote{cf. \cite[Proposition 2.3]{z06}.}
    Let $X$ be a compact connected $n$-dimensional Riemannian manifold, $n \geq 2$ and $f: X \to \R$ be an $L$-Lipschitz excellent Morse function. Let $p \in X$ and $\epsilon_p:=||f - d(p,\cdot) ||_\infty$. Let $X_f$ be the Reeb graph of $f$ and $\pi: X \to X_f$ be the Reeb quotient map. Then the metric distortion of $\pi$ is less than or equal to
\begin{align*}&2(\beta_1(X)+1)^2\Bigg( \left( \frac{2^{n+1}\,L}{(\beta_1(X)+1)}\cdot \frac{\mu^n(X)}{\thick{f}} \right)^{\frac{1}{n}} + 16\bigg(\big(\diam(X)\big)^{\frac{1}{n}}\epsilon_p^{\frac{n-1}{n}}+\epsilon_p\bigg) \Bigg) 
+ |L-1|\,\diam(X),
\end{align*}
where $\mu^n$ denotes the $n$-dimensional Hausdorff measure on $X$.
\end{thm}

Note that if we replace $\epsilon_p$ in the statement by $\epsilon_p:=\inf_{c\in\R}\|f-d(p,\cdot)+c\|_\infty$ (which is a smaller quantity) the theorem is still correct. Indeed, if we take $g=f+c$ where $c$ is the constant realizing $\epsilon_p$, then $X \to X_f$ coincides with $X \to X_g$ as maps of metric spaces.

At first sight it may appear that our upper bound is not tight in the case where $f$ is constant, since in that case the actual distortion of any such map is $\diam(X)$. Indeed, note that since when $L=0$ we have $\epsilon_p\leq \diam(X)$, our bound becomes $\diam(X)\cdot \big(64(\beta_1(X)+1)^2+1\big)$. However, note that the constant function is not Morse and therefore our analysis does not apply to it.

    Let us interpret this upper bound as follows. By distributing the product over the sum and combining topological constants depending on the  first Betti number of $X$ properly as $C(X)$, $D(X)$, and $E(X)$, the bound can now be written as
    \[\underbrace{C(X)\,L^\frac{1}{n}\,\left(\frac{\mu^n(X)}{\thick{f}}\right)^{\frac{1}{n}}}_{\mathrm{(I)}} + \underbrace{\big(D(X)\,\big(\diam(X)\big)^{1/n}\,\epsilon_p^{(n-1)/n}+ E(X)\,\epsilon_p\big)}_{\mathrm{(II)}} + \underbrace{\diam(X)\,|L-1|}_{\mathrm{(III)}}, \]
    which consists of three parts. The last part $\mathrm{(III)}$ shows that we prefer values of the Lipschitz constant $L$ to be close to $1$, since otherwise the distortion becomes close to $\diam(X)$, which is the worst case distortion which we already know. The second term $\mathrm{(II)}$ shows that  smaller values of $\epsilon_p$ are favored in order to obtain a smaller upper bound. These two observations can be combined to say that the last two terms of our upper bound become smaller for \textit{distance-like} functions, where $|L-1|+\epsilon_p$ is considered as a first degree \textit{Sobolev type} distance of $f$ to $d(p,\cdot): X \to \R$. Note that by Proposition \ref{approx}, the distance function $d(p,\cdot): X \to \R$ can be approximated (in terms of $|L-1|+\epsilon_p$) by excellent Morse functions.  
  
Now let us look at the first term $\mathrm{(I)}$ of the bound:  $C(X)\,L^\frac{1}{n}\,\left(\frac{\mu^n(X)}{\thick{f}}\right)^{\frac{1}{n}}$. Because of its effect on the third part of the bound, we already mentioned that $L$ is preferred to be around $1$. The first part shows that smaller total volume guarantees a better bound, as is expected from the situation arising from considering the boundary of a tubular neighborhood around an embedded metric graph. In addition, it shows that small values of $\thick{f}$ cause a problem. In the case where $X$ is a surface, $\thick{f}\geq 1$ by Remark \ref{surface}, but we do not have such a lower bound for arbitrary dimensions. Hence, dimensions greater than $2$ require extra care in the selection of $f$: In addition to being distance like, we also want control on $\thick{f}$. In Section \ref{sec:calc}, we construct rich families where we have such control.

Theorem \ref{main} follows directly from Proposition \ref{distortion} and Proposition \ref{levelMain}.

In section \ref{invariant} we see that for $n=2$ it holds that  $\thick{f}$ is always greater than or equal to $1$. Then, since the distance function can be approximated with Morse functions nicely (see section \ref{sec:dt}), in 
the case $n=2$ Theorem \ref{main} gives us that the distortion of the Reeb quotient map is bounded above by $4\sqrt{2} (\beta_1(X)+1)^{3/2} \cdot \big(\mathrm{area}(X)\big)^{1/2}$,  which is the main result of \cite{z06}.

    The novelty in our paper is twofold. Firstly, we extend some topological results from \cite{z06} which are easy in dimension two to any dimension $n \geq 2$, where much heavier use of differential topology is required. Secondly, we introduce a new metric invariant $\thick{f}$ which does not appear in dimension two since it is bounded from below by 1 in that case. We study this invariant for a class of metric spaces which we call thickened graphs.
    
\section{Preliminary results}

In this section we collect some results about metric spaces and differential topology that we invoke in the proofs of our main statements.
\subsection{Metric Spaces}

\begin{lemma}\label{setCover}
Let $\mathcal{S}$ be a family of sets. Then the following are equivalent.
\begin{enumerate}[i)]
\item For each $S,S'$ in $\mathcal{S}$, there exists $S=S_0,\cdots,S_n=S'$ such that $S_i \cap S_{i-1} \neq \emptyset$ for each $i=1,\dots,n$.

\item For each non trivial partition $\mathcal{S}=\mathcal{T} \amalg \mathcal{T'}$ we have \[ \Bigg(\bigcup_ {T \in \mathcal{T}} T\Bigg) \bigcap \Bigg(\bigcup_{T' \in \mathcal{T'}} T'\Bigg) \neq \emptyset.  \]
\end{enumerate}
\end{lemma}

\begin{proof}
$``i) \implies ii)"$ Take $S \in \mathcal{T}$ and $S' \in \mathcal{T'}$. Let $S_0,\dots,S_n$ be as in $i)$. Let $k$ be the maximal index such that $S_k \in \mathcal{T}$. Note that $k \neq n$ and $S_{k+1} \in \mathcal{T'}$. Then the intersection in $ii)$ contains $S_k \cap S_{k+1}$ which is nonempty.

$``ii) \implies i) "$ Let $\mathcal{S}_0:=\{ S \}$ and inductively define $$\mathcal{S}_{i+1}:=\{ T \in \mathcal{S}: T \cap (\cup_ {T' \in S_{i}} T') \neq \emptyset \}.$$ 
Let $\mathcal{T}=\cup_i\mathcal{S}_i$ and $\mathcal{T}'$ be its complement. The intersection in $ii)$ becomes empty by  definition of $\mathcal{S}_i$. Hence $\mathcal{T'}$ is empty and  $\mathcal{T}=\mathcal{S}$. Note that for $i \geq 1$, for each $T \in \mathcal{S}_i$. there exists a $T' \in \mathcal{S}_{i-1}$ such that $T \cap T' \neq \emptyset$. Since $S'$ is in $\mathcal{S}_i$ for some $i$, by using the property mentioned above we can construct a chain of sets as in $i)$.  
\end{proof}

\begin{lemma}\label{chainDiam}
Let $\mathcal{S}$ be a collection of subsets of a metric space as in Lemma \ref{setCover} $i)$. Then \[ \diam\Bigg(\bigcup_{S \in \mathcal{S}} S\Bigg) \leq \sum_{S \in \mathcal{S} } \diam(S). \]
\end{lemma}
\begin{proof}
Let $x \in S, y \in S'$ for some $S,S' \in \mathcal{S}$. Let $S=S_0,\dots,S_n=S'$ be as in Lemma \ref{setCover} $i)$. Let $x_i \in S_i \cap S_{i-1}$ for $i=1,\dots,n$. Let $x_0=x,x_{n+1}=y$. Then we have
\[d(x,y)=d(x_0,x_{n+1}) \leq \sum_{i=0}^n d(x_i,x_{i+1}) \leq \sum_{i=0}^n \diam(S_i).  \]
\end{proof}

\begin{prop}\label{closedCover}
Let $X$ be a connected metric space and $\mathcal{A}$ be a finite closed cover of $X$. Then 
\[\diam(X) \leq \sum_{A \in \mathcal{A}} \diam(A). \]
\end{prop}
\begin{proof}
Without loss of generality we can assume that the sets in $\mathcal{A}$ are nonempty. It is enough to show that $\mathcal{A}$ is a cover as in Lemma \ref{setCover} $ii)$, since then the result follows from Lemma \ref{setCover} and Lemma \ref{chainDiam}.

Let $\mathcal{A}=\mathcal{B} \cup \mathcal{B'}$ be a nontrivial partition of $\mathcal{A}$. Then $C:=\cup_{B \in \mathcal{B}} B$ and $C':=\cup_{B' \in \mathcal{B'}} B'$ are nonempty closed subsets of $X$ such that $X=C \cup C'$. Since $X$ is connected, $C \cap C'\neq \emptyset$.
\end{proof}

For an integer $k \geq 0$, we denote the $k^{th}$ Hausdorff measure  \cite{bbi01} on a metric space by $\mu^k$.
\begin{prop}[Coarea Formula]\label{coarea}
If $f: X \to \R$ is a smooth $L$-Lipschitz function from an n-dimensional Riemannian manifold $X$, then for each $t_0 \leq t_1 \in \R$ we have
\[\mu^n(f^{-1}[t_0,t_1]) \geq \int_{t_0}^{t_1} \frac{\mu^{n-1}(f^{-1}(t))}{L}dt.\]

\end{prop}

\subsection{Differential Topology}\label{sec:dt}
A Morse function is called \textit{excellent} if there is a unique critical point for each critical value.

\begin{prop}\label{approx}
Let $f:X \to \R$ be an $L$-Lipschitz function from a compact Riemannian manifold. Then for each $\epsilon > 0$ and $\delta > 0$, there exists an $(L + \delta)$-Lipschitz excellent Morse function $f':X \to \R$ such that $||f-f' ||_\infty \leq \epsilon$.
\end{prop}
\begin{proof}
By \cite[Proposition 2.1.]{gw79}, there exists an $L$-Lipschitz smooth function $g: X \to \R$ such that $||f-g||_\infty \leq \epsilon/3$. By \cite[Proposition 1.2.4.]{ad14} there exists a Morse function $g': X \to \R$ such that $g'$ is $(L+(\delta/2))$-Lipschitz and $||g-g'||_\infty \leq \epsilon/3$. As it is explained in \cite[p. 40]{ad14}, there exists an excellent Morse function $f': X \to \R$ such that $f$ is $(L+\delta)$-Lipschitz and $||g'-f'||_\infty \leq \epsilon/3$. This $f'$ satisfies the desired properties.
\end{proof}

When $f: X \to \R$ is a Morse function defined on a compact manifold, then we already know that the Reeb space is a graph. The following proposition shows that when $f$ is an excellent Morse function, the Reeb graph satisfies an extra condition. We will use this result later in the proof of Proposition \ref{distortion}.

\begin{prop}\label{degree}
Let $X$ be a compact Riemannian manifold of dimension $n \geq 2$, and let $f: X \to \R$ is a proper excellent Morse function. Then the vertices of the Reeb graph $X_f$ have degrees at most 3.
\end{prop}
\begin{proof}
Consider the Reeb graph $X_f$ as a directed graph, where the orientation of each edge is given so that $f$ is increasing. Since $f$ is an excellent Morse function, the vertices of the Reeb graph $X_f$ corresponds to the critical points of $f$. Let $p$ be such a point. Let $W^s_p,W^u_p$ be the stable and unstable manifolds of the gradient of $f$ at $p$ (see \cite[p. 27]{ad14}). They are disks meeting at $p$ and the sum of their dimensions is $n$ \cite[Proposition 2.1.5]{ad14}. Number of incoming edges (resp. outgoing edges) is more than one only if $W^s$ (resp. incoming) has dimension $1$, and in that the number of edges is at most $2$. This implies our statement for $n \geq 3$. For $n=2$, the case where  both the stable and the unstable manifold are 1-disks is realized by a \textit{saddle}, and the degree of the corresponding vertex then becomes $3$.
\end{proof}

\begin{lemma}\label{codim1}
Let $Y$ be codimension 1 compact submanifold of $X$. If $Y$ has at least $\beta_1(X)+1$ connected components, then $X - Y$ is disconnected.
\end{lemma}
\begin{proof}
During the proof we use $\Z/2$ coefficients for cohomology. By the long exact cohomology sequence for the pair $(X,X-Y)$, we have the following exact sequence:
\[ H^0(X,X-Y) \to H^0(X) \to H^0(X-Y) \to H^1(X,X-Y) \to H^1(X). \]
For any open neighborhood $N$ of $Y$, by excision $H^*(X,X-Y) \cong H^*(N,N-Y)$. Let $N$ be a tubular neighborhood of $Y$ in $X$, hence $(N,N-Y)$ is homeomorphic to $(N_X Y,N_X Y - Y)$ where $N_X Y$ is the rank 1 normal bundle of $Y$ in $X$ \cite[p. 100]{b93}. By the Thom Isomorphism Theorem \cite[p. 106]{ms74}, $H^1(N_X,N_X-Y) \cong H^0(Y)$, $H^0(N_X Y,N_X - Y)=0$. Hence the exact sequence above becomes
\[ 0 \to H^0(X) \to H^0(X-Y) \to H^0(Y) \to H^1(X), \]
which implies
\[ \dim H^0(Y) \leq \dim H^0(X-Y) - 1 + \beta_1(X).  \]
\end{proof}

\begin{prop}\label{seperate}
Let $X$ be a connected manifold, $f: X \to \R$ be a proper Morse function, $t_0$ be a regular value of $f$ and $Y=f^{-1}(t_0)$. Let $s < t_0 < t $, $A$ be a connected component of $f^{-1}(s)$ and $B$ be a connected component of $f^{-1}(t)$ . Then there exists connected components $Y_1,\dots,Y_k$, $k \leq \beta_1(X)+1$ of $Y$ separating $A$ from $B$.
\end{prop}
\begin{proof}
Since $s < t_0 < t $, $Y$ separates $A$ from $B$. Let $Y_1,\dots,Y_k$ be a minimal collection of components of $Y$ separating $A$ from $B$. We will show that $k \leq \beta_1(X)+1$. Note that if $k > \beta_1(X) + 1$, then the family $Y_2,\dots, Y_k$ separates $X$ into at least two components by Lemma \ref{codim1}. Hence it is enough to show that this family does not separate $X$.

Let $Y'=Y_1 \cup \dots \cup Y_k$. Let us show that $X-Y'$ has exactly two components corresponding to $A$ and $B$. It has at least two components corresponding to $A$ and $B$. By minimality, there exist a curve $\gamma_i$ from $A$ to $B$ intersecting $Y_i$ and not intersecting $Y_j$ for each $i,j=1,\dot,k$, $i \neq j$. Choose closed tubular neighborhoods $N_i$ of $Y_i$ for each $i=1,\dots,k$ such that they are pairwise disjoint and $N_i$ does not intersect $\gamma_j$ for $i \neq j$. The normal bundle $N_{Y_i} X$ is trivial since it is 1-dimensional and it has a nowhere zero section given by the gradient of $f$. Therefore $N_i - Y_i$ has exactly two components. We can modify $\gamma_i$ if necessary so that it intersects $Y_i$ once and transversally. Let $x,y$ be the entrance and exit points of $\gamma_i$ to $N_i$ respectively. Note that $x,y$ should be in different components of $N_i - Y_i$ since otherwise we can modify $\gamma_i$ so it avoids $N_i$, hence $Y'$ altogether. Now, remove the part of $\gamma_i$ between $x,y$ and by using triviality of the normal bundle replace it with one that intersects $Y_i$ at a single point transversally. Let $x_i$ be this point. $\gamma_i - \{x_i \}$ has two components, each intersecting with different components of $N_i - Y_i$ by transversality. Therefore $N_i-Y_i=N_i^A \amalg N_i^B$ where $N_i^A$ is the component of $N_i - Y_i$ which is in the same component of $X - Y'$ with $A$, and $N_i^B$ is defined similarly. Now for any point in $q$ let $\gamma$ be a path in $X$ connecting $q$ to $A$. If $\gamma$ does not intersect $Y'$ then $q$ is in the same component with $A$ in $X - Y'$. Assume that $\gamma$ intersects $Y'$. Let $N_i$ be the first tubular neighborhood that $\gamma$ enters into. Then $q$ is in the same component of $X - Y'$ with the entrance point, which we have shown that either connected to $A$ or $B$ in $X - Y'$. Therefore $X - Y'$ has two components.

Any point in $Y_1$ can be connected by a path to both $B$ and $A$ by using $\gamma_i$ in $X - (Y_2 \cup \dots \cup Y_k)$. Combining this with the previous paragraph, we see that $X - (Y_2 \cup \dots \cup Y_k)$ is connected.

\end{proof}    
    
\section{Reeb Graphs}

Given a function $f: X \to Y$ from a topological space $X$, the \textit{Reeb space} $X_f$ corresponding to $f$ is defined as the quotient space $X/ \sim$ where $x \sim y$ if $x,y$ are in the same connected component of a fiber of $f$. Note that $f$ splits through the canonical surjection $\pi: X \to X_f$ which we call the Reeb quotient map. By an abuse of notation we denote this map by $f$ too, i.e. $f: X_f \to Y$ and $f \circ \pi = f$.

If $Y$ is a metric space, the Reeb space has a metric structure defined as follows. For a path $\gamma \in X$, define $l_f(\gamma) := l_Y(f \circ \gamma)$, where $l_Y$ is the length structure of the metric space $Y$ \cite{bbi01}. Let $d_f$ denote the pseudo-metric induced by this length structure. Note that if $x \sim y$ then $d_f(x,y)=0$, hence $d_f$ induces a pseudo-metric on the Reeb space.

In this paper we are interested in the case where $Y = \R$ and $f$ is an excellent Morse function from a compact Riemannian manifold $X$. In that case the Reeb space becomes a graph, which we denote by $X_f$, and the metric $d_f$ becomes a length metric on $X_f$. Note that we have the following diagram:
\[\begin{tikzcd}
X \arrow{r}{\pi} \arrow{rd}{f} & X_f \arrow{d}{f} \\
&\R.
\end{tikzcd}\]
The following proposition gives some basic facts about the topological behavior of the Reeb quotient map, which we state without a proof.
\begin{prop}\label{reeb}
\begin{enumerate}[i)]
\item $\pi_*: \pi_1(X) \to \pi_1(X_f)$ is surjective. In particular $\beta_1(X_f) \leq \beta_1(X)$.
\item \cite[Proposition 1.2.]{z06} The preimage of a connected set under the map $\pi: X \to X_f$ is connected.
\end{enumerate}
\end{prop}
Now we can prove the following proposition.
\begin{prop}\label{lipschitz}
\begin{enumerate}[(i)]
\item \cite[Proposition 1.1]{z06}\footnote{We give a different proof.} If $f: X \to \R $ is $L$-Lipschitz, then $\pi: X \to X_f$ is $L$-Lipschitz too.
\item $f: X_f \to \R$ is $1$-Lipschitz.
\end{enumerate}
\end{prop}
\begin{proof}
Given points $x,y$ in $X$, let $\Gamma(x,y)$ denote the set of continuous curves from $x$ to $y$ in $X$. Let $l_\R,l_X$ denote the canonical length structures on $\R, X$ respectively.

$(i)$ Let $x,y$ be points in $X$ and $r,s$ be their images under $\pi$ respectively. Then we have
\[d_f(r,s)= \inf_{\gamma \in \Gamma(x,y)} l_\R(f \circ \gamma) \leq L \inf_{\gamma \in \Gamma(x,y)} l_X (\gamma) = L\, d(x,y). \]

$(ii)$ Let $r,s$ be points in $X_f$ and $x,y$ be points in $X$ mapped onto $r,s$ under $\pi: X \to X_f$ respectively. Then we have
\[|f(r)-f(s)|=|f(x)-f(y)| \leq \inf_{\gamma \in \Gamma(x,y)} l_\R(f \circ \gamma) = d_f(r,s). \]
\end{proof}

\section{Distortion of the Reeb quotient map}

Recall that the distortion of a function $\phi: X \to Y$ between metric spaces $(X,d), (Y,d')$ is defined as $$\mathrm{dis}(\phi):=\sup_{x,x' \in X} |d(x,x')-d'(\phi(x),\phi(x'))|.$$ If $\phi$ is surjective, then the Gromov-Hausdorff distance between $X$ and $Y$ satisfies $\dgh(X,Y) \leq \mathrm{dis}(\phi)/2$ \cite{bbi01}.

In this section we  follow \cite[Section 3]{z06} with suitable modifications to reach the level of generality we want.
    
    Let $X$ be a compact Riemannian manifold and $f: X \to \R$ be an $L$-Lipschitz excellent Morse function. Let $p$ be a point in $X$ and let $\theta:=\pi(p) \in X_f$. Let $\epsilon_p:=||f-d(p,\cdot)||_\infty$.
    
    Given a subset $Y$ of $X_f$ let us denote its preimage in $X$ under $\pi$ by $C_Y$. For $r \in X_f$ let $C_r:=C_{\{r\}}$ and $x \in X$ let $C_x:=C_{\pi(x)}$. Note that if $Y$ is connected then $C_Y$ is connected by Proposition \ref{reeb}. In particular for any $x$ in $X$ and any $r \in X_f$, $C_x$ and $C_r$ are connected. 
    
    Given three subspaces $A,B,C$ of a topological space we say $B$ separates $A$ and $C$ if every path from $A$ to $C$ intersects $B$.
    
    \begin{prop}\label{comparison1}\cite[Lemma 3.6.]{z06}
    Let $r,s$ be points in $X_f$ so that $s$ separates $r$ and $\theta$. Then for any $x$ such that $\pi(x)=r$, we have 
    \[\mathrm{dist}(x,C_s) \leq d_f(r,s) + 2 \epsilon_p. \]
    \end{prop}
    \begin{proof}
    Since $s$ separates $r$ and $\theta$, $C_s$ separates $C_r$ from $p$. Let $x$ be a point in $C_r$, i.e. $\pi(x)=r$. Then there is a point $y$ in $C_s$ such that there is a length minimizing geodesic from $x$ to $p$ containing $y$. Hence we have
    \begin{align*}
    \mathrm{dist}(x,C_s) &\leq d(x,y) = d(x,p) - d(y,p) \\
    					 &\leq f(x)+f(y) + 2\epsilon_p \\
                         &=f(r)-f(s) + 2\epsilon_p \leq  d_f(r,s) + 2\epsilon_p \,\,\, \mbox{(By Proposition \ref{lipschitz}).}
    \end{align*}
    \end{proof}
    
    \begin{prop}\label{comparison2}\cite[Lemma 3.8.]{z06}
    Let $r,s$ be points in $X_f$ and $\gamma$ be a continuous curve between $r,s$ such that ${r,s}$ separates $\gamma$ from $\theta$. Then we have
    \[\mathrm{dist}(C_r,C_s) \leq \mathrm{length}(\gamma) + 4\epsilon_p. \]
    \end{prop}
    \begin{proof}
    $C_\gamma$ is a connected compact subset of $X$. Define $F_r$ to be the subset of $C_\gamma$ consisting of points such that there is a length minimizing geodesic from that point to $p$ intersecting $C_r$. Define $F_s$ similarly. Note that $F_r,F_s$ are closed and they are nonempty since they contain $C_r,C_s$ respectively. Furthermore, they cover $C_\gamma$ since $C_r \cup C_s$ separates $C_\gamma$ and $p$. Hence their intersection is nonempty. Let $x \in F_r \cap F_s$. There exists $y \in C_r$ (resp. $y' \in C_s$) such that there is a length minimizing geodesic from $x$ to $p$ containing $y$ (resp. $y'$). Now we have
    \begin{align*}
    \mathrm{dist}(C_r,C_s) &\leq d(y,y') \\
    					   &\leq d(x,y) + d(x,y') = d(x,p)-d(y,p) + d(x,p)-d(y',p) \\
                           &\leq f(x)-f(y) + f(x) - f(y') + 4\epsilon_p = f(\pi(x))-f(r) + f(\pi(x))-f(s) + 4\epsilon_p\\
                           &\leq d_f(r,\pi(x))+d_f(\pi(x),s) + 4\epsilon_p \leq \mathrm{length}(\gamma) + 4\epsilon_p.
    \end{align*}
    \end{proof}
    
    \begin{prop}\label{split}\cite[Lemma 3.9]{z06}
Let $(G,\theta)$ be a pointed finite topological graph, whose vertices have degrees at most three. Let $\gamma$ be a simple path from a point $r$ to a point $s$ in $G$. Then there are points $r_1,\dots,r_n$, $n \leq 2\,\beta_1(G)+1$ on $\gamma$ such that $r_0$ separates $r$ from $\theta$, $r_n$ separates $s$ from $\theta$ and $\{r_i,r_{i+1}\}$ separates the interior of the segment of $\gamma$ from $r_i$ to $r_{i+1}$ from $\theta$.
	\end{prop}
    
    \begin{prop}\label{distortion}
    Assume that the level sets of $\pi: X \to X_f$ have diameter less than or equal to a constant $C>0$. Then
    \[\mathrm{dis}(\pi) \leq (2\beta_1(X)+1)(C + 4\epsilon_p)+ |L-1|\,\diam(X). \]
    \end{prop}
    \begin{proof}
    Let $x,y$ be points in $X$ and $r,s$ be their images under $\pi$ respectively. By Proposition \ref{lipschitz} we have
    \[d_f(r,s)-d(x,y) \leq (L-1)\,d(x,y) \leq |L-1|\,\diam(X). \]
    Let $\gamma$ be a length minimizing curve between $r,s$. By Proposition \ref{degree}, $X_f$ is a graph satisfying the conditions of Proposition \ref{split}, hence we can choose $r_1,\dots,r_n$ as in Proposition \ref{split}. Let $\gamma_i$ denote the part of $\gamma$ between $r_i,r_{i+1}$ for $1 \leq i < n$. Note that $d_f(r_i,r_{i+1})=\mathrm{length}(\gamma_i)$. Now, by Proposition \ref{comparison1} and \ref{comparison2} we have
    \begin{align*}
    d_f(r,s)&=\mathrm{length}(\gamma) = d_f(r,r_0) + \sum_{i=1}^{n-1} \mathrm{length}(\gamma_i) + d_f(r_n,s) \\
    	  &\geq \mathrm{dist}(x,C_{r_1})-2\epsilon_p + \sum_{i=1}^{n-1}(\mathrm{dist}(C_{r_i},C_{r_{i+1}})-4\epsilon_p) + \mathrm{dist}(C_{r_n},s)-2\epsilon_p \\
          &\geq d(x,y) - n\,C - 4\,n\,\epsilon_p \geq d(x,y) - (2\beta_1(X)+1)(C+4\epsilon_p).
    \end{align*}
    \end{proof}

\section{Thickness}\label{invariant}

Given $k\in\mathbb{N}$,  we define the $k^{th}$-\textit{thickness} of a path connected metric space $A$ by
\[\thick{A}^k:= \frac{\mu^k(A)}{\big(\diam(A)\big)^k}. \]
Define it as $1$ if the diameter is $0$. Note that low values of this invariant intuitively indicate that the space has very thin parts: That is, parts with large diameter but small $k^{th}$ dimensional Hausdorff measure. For a metric space with several path components, we define the thickness by taking the minimum thickness across its components. We will elaborate on the interpretation of the thickness invariants thorough the examples that follow and the situation described in section \ref{sec:calc}.

\begin{example}\label{sphere}
Let $s_n$ be the total $n^{th}$ Hausdorff measure of the standard $n$-dimensional unit sphere. It is known that $$s_n=\frac{2 \pi^{(n+1)/2}}{\Gamma((n+1)/2)}.$$ Let $\mathbb{S}^n(r)$ be an $n$-dimensional sphere of radius $r$ with the standard Riemannian metric. Then $\mu^n(\mathbb{S}^n(r))=s_n r^n$ and $\diam(\mathbb{S}^n(r))=\pi \,r$. Therefore, $$\thick{\mathbb{S}^n(r)}^n=\frac{s_n}{\pi^n}=\frac{2 \pi^{-(n-1)/2}}{\Gamma((n+1)/2)},$$ which is independent of $r$.
\end{example}

\begin{remark}
$\thick{A}^1 \geq 1$ since length is always greater than or equal to the diameter.
\end{remark}

Given a smooth function $f:X \to \R$ on an $n$-dimensional manifold, we define
\[\thick{f}:=\inf_{t \in \mathrm{range}(f)} \thick{f^{-1}(t)}^{n-1}. \]

\begin{example}\label{constant}
If $f$ is the constant function, then the numerator in the definition of thickness invariant on the unique level set becomes infinity, hence $\thick{f}=\infty$. 
\end{example}

\begin{remark}\label{surface}
If $X$ is a surface then $\thick{f} \geq 1$ for each $f: X \to \R$.
\end{remark}

\subsection{A calculation: Thickened filtered graphs}\label{sec:calc}

Here we first introduce \textit{thickened filtered graphs} and then find a lower bound for the thickness of this kind of filtered spaces. We first need to define how to \textit{glue} filtered spaces.

\begin{figure}
	\centering
    \includegraphics[width=0.3\textwidth]{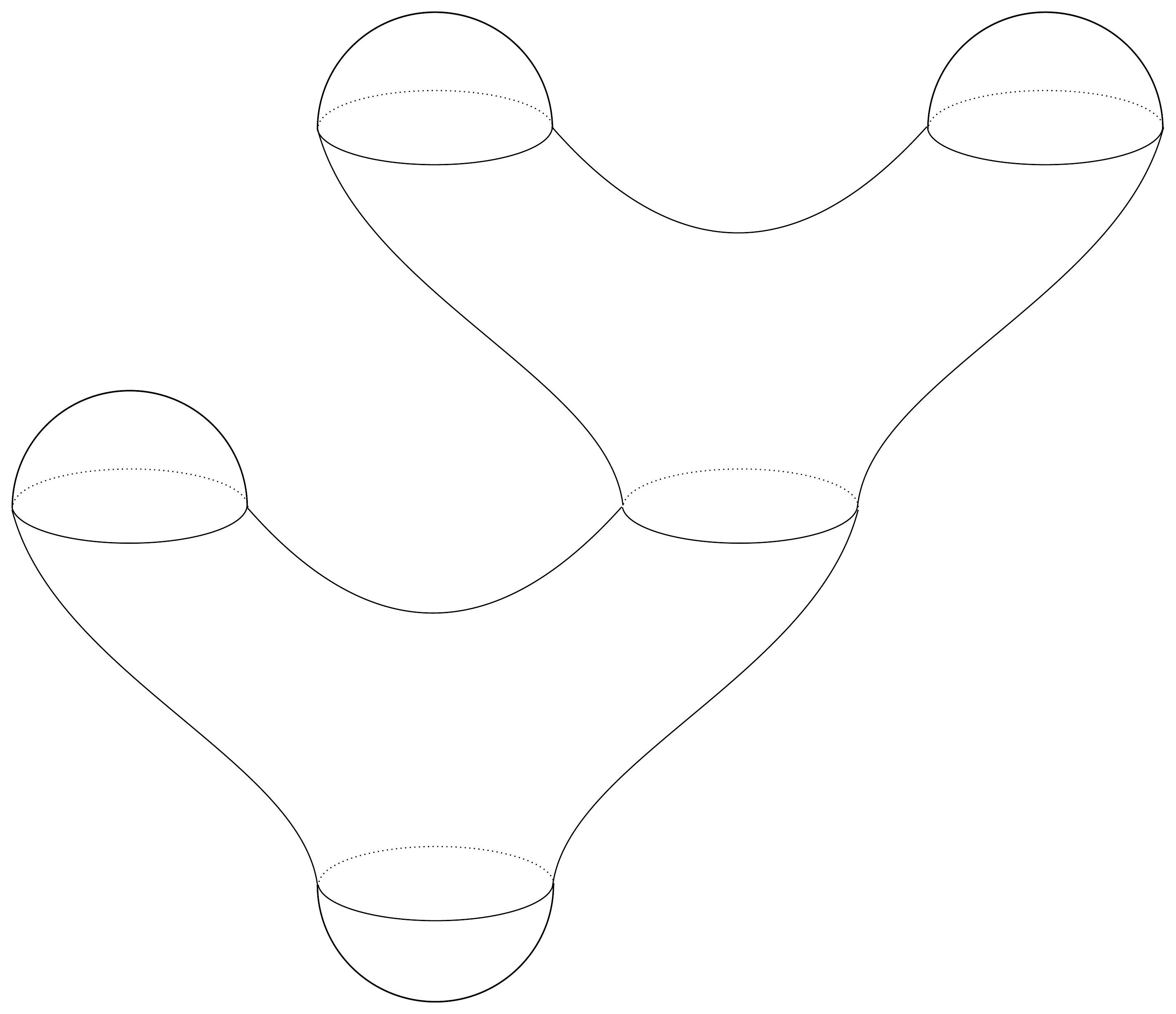}
\caption{A 2-dimensional thickened filtered graph.}
\label{fig:graph}
	\end{figure} 

A compact filtered space $f:X \to \R$ has a maximum and minimum level sets which we denote by $X_t,X_b$ respectively. Let us denote the value of $f$ on $X_t,X_b$ by $t_f,b_f$ respectively. If we are given two filtered spaces $f: X\to \R$ and $g: Y \to \R$ together with a homeomorphism $\varphi: X_t \simeq Y_b$, we can glue this two topological spaces through $\varphi$ and then define a filtration $h: X \coprod_\varphi Y \to \R$ by $h|_X:= f$ and $h|_Y := g -b_g + t_f$. See how the gluing is done along circles in Figure \ref{fig:graph}. Furthermore, if an addition $X$ and $Y$ are metric spaces endowed with metrics $d_X$ and $d_Y$, respectively, and the homeomorphism $\varphi:X_t\rightarrow Y_b$ is an isometry, then we define a metric $d$ on $X \coprod_\varphi Y$ as follows. The restriction of $d$ to $X$ is $d_X$, the restriction of $d$ to $Y$ is $d_Y$ and for all $x \in X, y \in Y$, $$d(x,y):= \inf_{z \in X_t} \big(d_X(x,z)+d_Y(\varphi(z),y)\big).$$ 

Note that if $X,Y$ are length spaces, then so is $X \coprod_\varphi Y$. We then have:

\begin{lemma}\label{gluing-inv}
Let $h:Z \to \R$ be the filtered metric space which is obtained by gluing filtered metric spaces $f:X \to \R$, $g: Y \to \R$, described as above. Then $\thick{h}=\min(\thick{f},\thick{g})$.
\end{lemma}
\begin{proof}
Recall that $T_f$ is defined through taking a minimum through the levels sets of $f$. The result follows from the fact that the level sets of $h$ can be seen as the union of the level sets of $f,g$.
\end{proof}

Since the thickness of a filtered metric space is defined as a minimum along its connected components, we can do the gluing through a collection of connected components of the maximum level set of the first filtered space with a collection of connected components of the minimum level set of the second filtered space. Lemma \ref{gluing-inv} still holds.

\begin{figure}
	\centering
    \includegraphics[width=0.3\textwidth]{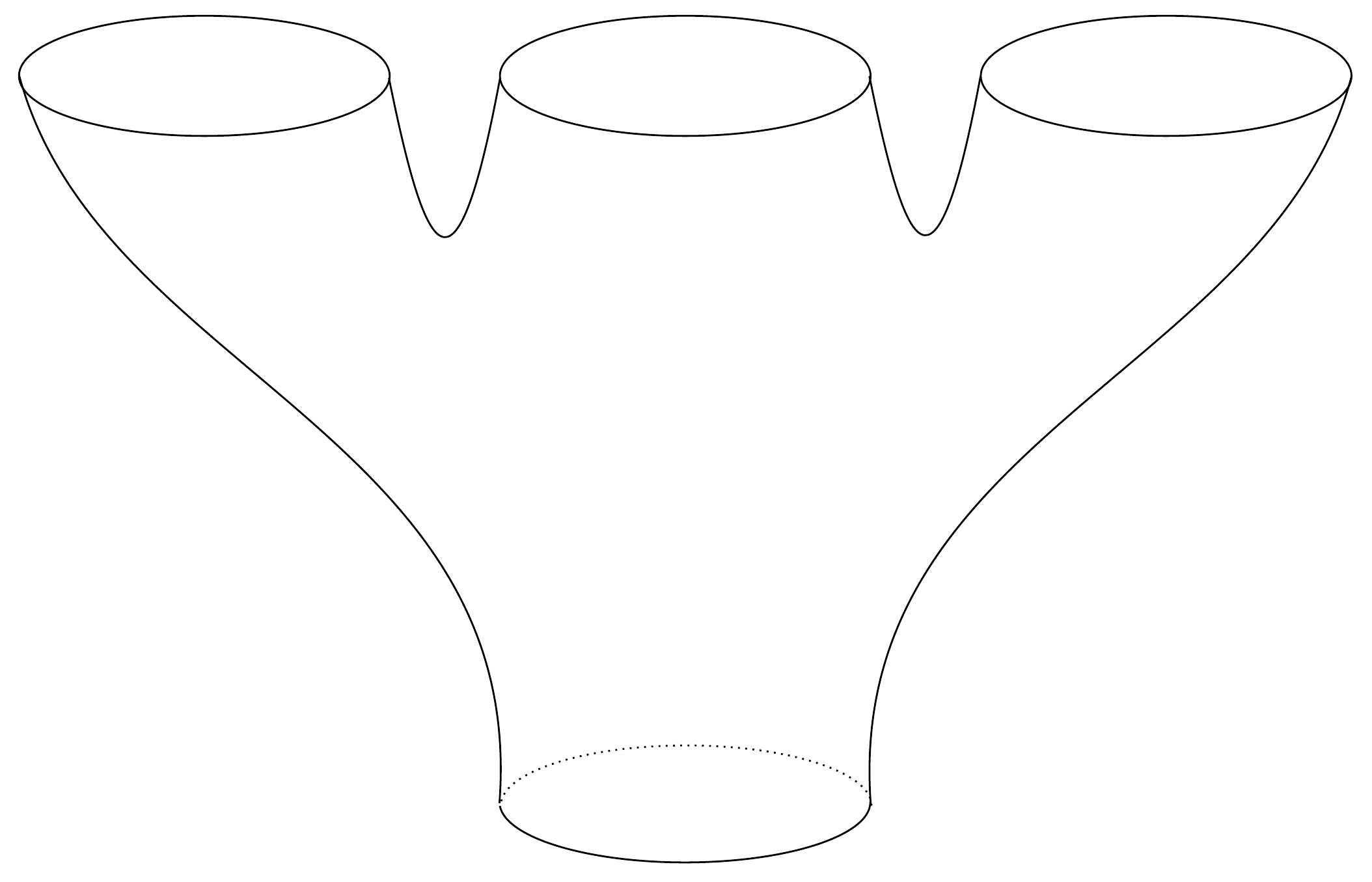}
\caption{A 2-dimensional thickened 3-fork.}
\label{fig:fork}
	\end{figure} 

A $k$-\textit{fork} is a directed graph with $k+2$ vertices $v_1,\dots,v_k,w,w'$ where the directed edges are given by $(w',w)$ and $(w,v_i)$ for all $i$. Note that its geometric realization can be embedded into $\R^2$ with coordinates $(x,y)$ where the height $y$ is increasing along the directed edges and the points level set are collinear. We call such a realization a \textit{filtered} $k$-fork. It can be thickened and smoothened in $\R^{n+1}$ with coordinates $(x=x_1,\dots,x_n,y)$ such that it becomes an $n$-dimensional submanifold of $\R^{n+1}$ and the level sets with respect to the height function $h$ is the boundary of the union of balls of same radius in $\R^n$ with coordinates $(x_1,\dots,x_n)$ where the centers of the balls are given by the corresponding level sets of the fork. We call such a metric space  an \textit{n-dimensional filtered thickened k-fork} (see Figure \ref{fig:fork}). 

\begin{lemma}\label{fork-inv}
Let $s_n$ be the total $n^{th}$ Hausdorff measure of the standard $n$-dimensional unit sphere. Then for any $n$-dimensional filtered thickened $k$-fork $f: F \to \R$ we have $$\thick{f} \geq \frac{s_{n-1}}{(k\,\pi)^{n-1}} = \frac{\thick{\mathbb{S}^{n-1}}^{n-1}}{k^{n-1}}.$$
\end{lemma}
\begin{proof}
By definition of a filtered thickened fork, a level set of $f$ is the boundary of the union of balls of radius $r$ with collinear centers in $\R^n$. Let $\mathcal{S}$ be a connected component of such a level set, and let $p_1,\dots,p_m$ be the centers given in a linear order. Note that by linearity, $\mathcal{S}$ contains the two half $(n-1)$-dimensional spheres with centers $p_1$ and $p_m$. Hence $\mu^{n-1}(\mathcal{S}) \geq \mu^{n-1}(\mathbb{S}^{n-1}(r))$. Since $m \leq k$, $\diam(\mathcal{S}) \leq k \, \diam(\mathbb{S}^{n-1}(r))$ by Proposition \ref{closedCover}. Therefore
$$\thick{\mathcal{S}}^{n-1} \geq \frac{\mu^{n-1}(\mathbb{S}^{n-1}(r))}{\big(k \, \diam(\mathbb{S}^{n-1}(r))\big)^{n-1}} = \frac{s_{n-1}}{(k\,\pi)^{n-1}} \,\, \mbox{(see Example \ref{sphere})}. $$
\end{proof}

\begin{figure}
	\centering
    \includegraphics[width=0.3\textwidth]{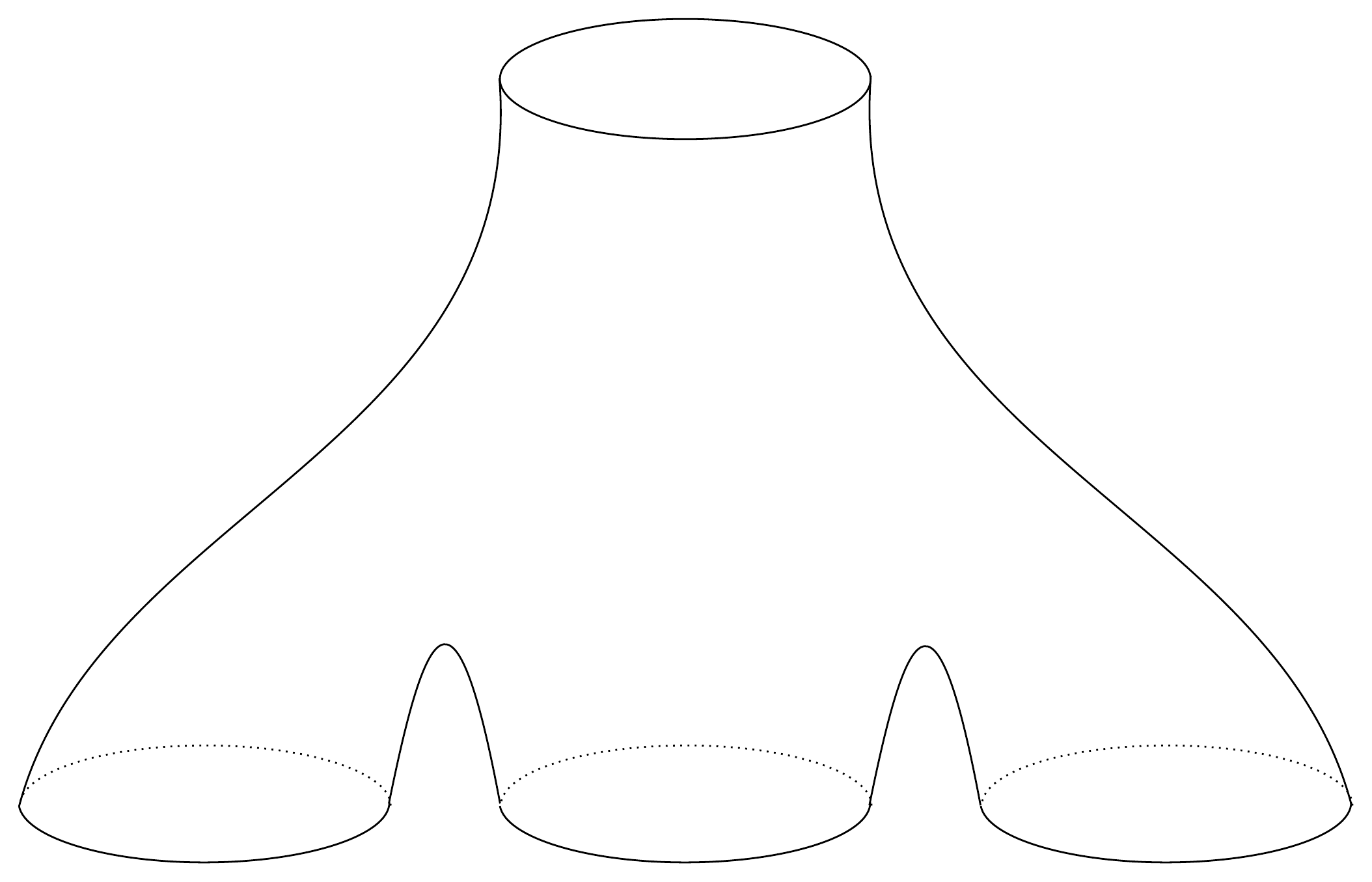}
\caption{An inverse 2-dimensional thickened fork.}
\label{fig:inverse}
	\end{figure} 

An \textit{inverse} filtered thickened fork is defined similarly but with the negative filtration (see Figure \ref{fig:inverse}). Negative filtration of a filtered space $f: X \to \R$ is defined as $-f:X \to \R$. Lemma \ref{fork-inv} still holds for this case, since taking the negative filtration does not change the thickness invariant.

\begin{dfn}(Thickened graph)
An $n$-dimensional filtered thickened graph is obtained by gluing $n$-dimensional thickened forks and inverse forks, and filtered half-spheres filtered by the height function in $\R^{n+1}$ (see Figure \ref{fig:graph}).
\end{dfn}
The following proposition follows from Lemma \ref{gluing-inv} and Lemma \ref{fork-inv}.
\begin{prop}\label{graphinv}
Let $f: \mathcal{G} \to \R$ be an $n$-dimensional filtered thickened graph and let $K$ be the maximal $k$ where a $k$-fork is used in the construction of $\mathcal{G}$. Then 
$$\thick{f} \geq \frac{\thick{\mathbb{S}^{n-1}}^{n-1}}{K^{n-1}}.$$
\end{prop}

\section{The diameter of a fiber of the Reeb quotient map}

Let $X$ be a compact connected $n$-dimensional Riemannian manifold, $n \geq 2$ and $f: X \to \R$ be a $L$-Lipschitz smooth function. For any $p \in X$, let $\epsilon_p:=||f - d(p,\cdot) ||_\infty$.

    \begin{lemma}\label{diamSum}\footnote{Cf. \cite[Lemma 2.2]{z06}.}
    Let $t_0$ be a regular value of $f$ and let $p \in X$ be a point such that $t_0>f(p)$. Let $t>t_0$ and $A$ be a connected component of $f^{-1}(t)$. Then we have
    \[\mu^{n-1}(f^{-1}(t_0)) \geq \frac{\thick{f}}{(\beta_1(X)+1)^{n-2}} \big( \diam(A) - 2(\beta_1(X)+1)(t-t_0 + 2\epsilon_p) \big)^{n-1}. \]
    \end{lemma}
    \begin{proof}
    By Proposition \ref{seperate} there are components $Y_1,\cdots,Y_k$, $k \leq \beta_1(X)+1$ of $Y:= f^{-1}(t_0)$ separating $A$ from $p$. Let $A_i$ be the subset of $A$ consisting of points which can be connected to $p$ through $Y_i$ by a length minimizing geodesic. Note that $\{A_i\}_{i=1,\dots,k}$ is a finite closed cover of $A$, hence by Proposition \ref{closedCover}
    \[\diam(A) \leq \sum_{i=1}^k \diam(A_i). \]
    
    For each $x$ in $A_i$, there is a point $x_i$ such that there is a length minimizing geodesic from $x$ to $p$ containing $x_i$. Given $x,y \in A_i$ we have
    \begin{align*}
    	d(x,y) &\leq d(x,x_i) + d(x_i,y_i) + d(y_i,y) \\
        	   &= d(x,p)-d(x_i,p) + d(x_i,y_i) + d(y,p)-d(y_i,p) \\
               &\leq f(x)-f(x_i) + 2\epsilon_p + \diam(Y_i) + f(y)-f(y_i) + 2\epsilon_p \\
               &= \diam(Y_i) + 2(t - t_0 + 2\epsilon_p),
    \end{align*}
    hence
    \begin{align*}
    \diam(A) &\leq \sum_{i=1}^k (\diam(Y_i) + 2(t-t_0 + 2\epsilon_p) ) \\
    		 &\leq 2(\beta_1(X)+1)(t-t_0 + 2\epsilon_p) + \sum_{i=1}^k \diam(Y_i).
    \end{align*}
    Now we have
    \begin{align*}
    \mu^{n-1}(f^{-1}(t_0)) &\geq \sum_{i=1}^k \mu^{n-1}(Y_i) \\
    					   &\geq \sum_{i=1}^k \thick{f}\, \diam(Y_i)^{n-1} \\
                           &\geq \frac{\thick{f}}{k^{n-2}} \Bigg(\sum_{i=1}^k \diam(Y_i) \Bigg)^{n-1} \,\,\mbox{(By convexity of $\lambda \mapsto \lambda^{n-1}$ over $[0,\infty)$)} \\
                           &\geq \frac{\thick{f}}{(\beta_1(X)+1)^{n-2}} \big( \diam(A) - 2(\beta_1(X)+1)(t-t_0 + 2\epsilon_p) \big)^{n-1}.    \end{align*}
    \end{proof}
    
    \begin{prop}\label{levelMain}\footnote{Cf. \cite[Proposition 2.3]{z06}.}
    Let $A$ be a connected component of $f^{-1}(t)$. Then for any $p$ in $X$, the diameter of $A$ is less than or equal to
    \[\max \Bigg(8(\beta_1(X)+1)\,\epsilon_p, \left( \frac{2^{n+1}L(\beta_1(X)+1)^{n-1}\mu^n(X)}{\thick{f}} \right)^{\frac{1}{n}} + 8(\beta_1(X)+1)\,\big(\diam(X)\big)^{\frac{1}{n}}\,\epsilon_p^{\frac{n-1}{n}}\Bigg). \]

    \end{prop}
    \begin{proof}
    Assume $\diam(A) > 8(\beta_1(X)+1)\,\epsilon_p$. Let $\delta= \frac{\diam(A)}{4(\beta_1(X)+1)}$. Note that $\diam(A) \leq 2\sup_{x \in A} d(x,p) \leq 2t + 2\epsilon_p.$ Then we have
    \begin{align*} 
    t - \delta &\geq \frac{\diam(A)}{2} - \epsilon_p - \frac{\diam(A)}{4(\beta_1(X)+1)} \\
    		   &\geq \frac{\diam(A)}{4} - \epsilon_p> 2(\beta_1(X)+1)\epsilon_p - \epsilon_p \geq \epsilon_p \geq f(p),
    \end{align*}
    hence for any regular value $s \in [t-\delta,t]$ we can apply Lemma \ref{diamSum} to obtain a lower bound for $\mu^{n-1}(f^{-1}(s))$. Note that for such an $s$
    \begin{align*}
    \diam(A)-2(\beta_1(X)+1)(t-s+2\epsilon_p) &\geq \diam(A)-2(\beta_1(X)+1)(\delta+2\epsilon_p) \\
    &= \frac{\diam(A)}{2} - 4(\beta_1(X)+1)\,\epsilon_p > 0,
    \end{align*}
    Hence by Lemma \ref{diamSum} we get 
    \[\mu^{n-1}(f^{-1}(s)) \geq \frac{\thick{f}}{(\beta_1(X)+1)^{n-2}} \left(\frac{\diam(A)}{2} - 4(\beta_1(X)+1)\,\epsilon_p \right)^{n-1}.\]
    
    Since almost all values of a smooth function are regular, by the coarea formula (Proposition \ref{coarea}), we get
    \[L\,\mu^n(X) \geq \int_{t-\delta}^t \mu^{n-1}(f^{-1}(s)) \,ds\geq \frac{\thick{f}\,\diam(A)}{4(\beta_1(X)+1)^{n-1}} \left(\frac{\diam(A)}{2} - 4(\beta_1(X)+1)\,\epsilon_p \right)^{n-1}.\]
    This gives us
    \[\diam(A) \leq  \left( \frac{2^{n+1}L(\beta_1(X)+1)^{n-1}\mu^n(X)}{\thick{f}\,\diam(A)} \right)^{\frac{1}{n-1}} + 8(\beta_1(X)+1)\,\epsilon_p,\]
    which implies
    \[\big(\diam(A)\big)^{\frac{n}{n-1}} \leq \left( \frac{2^{n+1}L(\beta_1(X)+1)^{n-1}\mu^n(X)}{\thick{f}} \right)^{\frac{1}{n-1}} + 8(\beta_1(X)+1)\,\big(\diam(X)\big)^{\frac{1}{n-1}}\,\epsilon_p,\]
    hence
    \[\diam(A) \leq \left( \frac{2^{n+1}L(\beta_1(X)+1)^{n-1}\mu^n(X)}{\thick{f}} \right)^{\frac{1}{n}} + 8(\beta_1(X)+1)\,\big(\diam(X)\big)^{\frac{1}{n}}\,\epsilon_p^{\frac{n-1}{n}}. \]
    Note that here we used the fact that for $x,y$ positive and $0 \leq r \leq 1$, we have $(x+y)^r \leq x^r + y^r$.
    \end{proof}
   

\end{document}